\documentclass[10pt]{amsart}
     \makeatletter
     \def\section{\@startsection{section}{1}%
     \z@{.7\linespacing\@plus\linespacing}{.5\linespacing}%
     {\bfseries
     \centering
     }}
     \def\@secnumfont{\bfseries}
     \makeatother
\newtheorem{theorem}{Theorem}[section]
\newtheorem{lemma}[theorem]{Lemma}

\theoremstyle{definition}

\theoremstyle{remark}

\numberwithin{equation}{section}
\usepackage{amsmath,amsthm,amssymb,amsbsy}

\newcommand{\DD}{\mathbb{D}}

\newcommand{\RR}{\mathbb{R}}

\newcommand{\AAA}{\mathcal{A}}

\newcommand{\FFF}{\mathcal{F}}

\newcommand{\LLL}{\mathcal{L}}
\newcommand{\MMM}{\mathcal{M}}

\newcommand{\VVV}{\mathcal{V}}

\newcommand{\ldb}{[\hspace{-1.5pt}[}                        

\newcommand{\df}[1]{\,\mathrm{d}#1}                         

\newcommand{\eps}{\varepsilon}                              

\newcommand{\locall}{\ell}                                    

\newcommand{\msq}{\MMM^2}                                     

\newcommand{\aaf}{\AAA}                                       
\newcommand{\aai}{\AAA^i}                                     
\newcommand{\aal}{\AAA^i_\locall}                             

\newcommand{\vvf}{\VVV}                                       
\newcommand{\vvi}{\VVV^i}                                     
\newcommand{\vvl}{\VVV^i_\locall}                             

\begin{document}

\setlength{\parindent}{0cm}
\setlength{\parskip}{0.5cm}

\title[Existence results in martingale theory]{Proving existence results in martingale theory using a subsequence principle}

\author{Alexander Sokol}

\address{Alexander Sokol: Institute of Mathematics, University of
  Copenhagen, 2100 Copenhagen, Denmark}
\email{alexander@math.ku.dk}
\urladdr{http://www.math.ku.dk/$\sim$alexander}

\subjclass[2000] {Primary 60G07; Secondary 60G44, 60H05}

\keywords{Martingale, Compensator, Quadratic variation, Stochastic integral}

\begin{abstract}
New proofs are given of the existence of the compensator (or dual predictable
projection) of a locally integrable c\`{a}dl\`{a}g adapted process of
finite variation and of the existence of the quadratic variation process for a c\`{a}dl\`{a}g local
martingale. Both proofs apply a functional analytic subsequence
principle. After presenting the proofs, we discuss their application in giving a
simplified account of the construction of the stochastic integral of a
locally bounded predictable process with respect to a semimartingale.
\end{abstract}

\maketitle

\noindent

\section{Introduction}

Assume given a filtered probability space $(\Omega,\FFF,(\FFF_t),P)$
satisfying the usual conditions, see \cite{PP}, Section I.1, for the
definition of this and other standard probabilistic concepts. For a
locally integrable c\`{a}dl\`{a}g adapted process $A$ with initial
value zero and finite variation, the compensator, also known as the
dual predictable projection, is the unique locally integrable
c\`{a}dlag predictable process $\Pi^*_pA$ with initial value zero and
finite variation such
that $A-\Pi^*_pA$ is a local martingale. For a c\`{a}dl\`{a}g local martingale $M$
with initial value zero, the quadratic variation process is the
unique increasing c\`{a}dl\`{a}g adapted process $[M]$ with initial value zero
such that $M^2-[M]$ is a local martingale and $\Delta [M]=(\Delta
M)^2$. In both cases, uniqueness is up to indistinguishability.

For both the dual predictable projection and the quadratic variation,
the proofs of the existence of these processes are among the most
difficult in classical martingale theory, see for example \cite{RW2},
\cite{HWY} or \cite{PP} for proofs. In this article, we give new
proofs of the existence of these processes. The proofs are facilitated
by the following lemma, first applied in this form to probability theory in
\cite{BSV}. We also give a short proof of the lemma.

\begin{lemma}
\label{lemma:Mazur}
Let $(X_n)$ be sequence of variables bounded in $\LLL^2$. There exists
a sequence $(Y_n)$ such that each $Y_n$ is a convex combination of a
finite set of elements in $\{X_n,X_{n+1},\ldots\}$ and $(Y_n)$ is
convergent in $\LLL^2$.
\end{lemma}
\begin{proof}
Let $\alpha_n$ be the infimum of $EZ^2$, where $Z$ ranges through all
convex combinations of elements in $\{X_n,X_{n+1},\ldots\}$,
and define $\alpha=\sup_n\alpha_n$. If $Z=\sum_{k=n}^{K_n}\lambda_k
X_k$ for some convex weights $\lambda_n,\ldots,\lambda_{K_n}$, we obtain $\sqrt{EZ^2}\le\sup_n\sqrt{EX_n^2}$, in particular we have $\alpha_n\le\sup_n EX_n^2$ and so
$\alpha\le\sup_n EX_n^2$ as well, proving that $\alpha$ is finite. For each $n$, there
is a variable $Y_n$ which is a finite convex combination of elements
in $\{X_n,X_{n+1},\ldots\}$ such that $E(Y_n)^2\le \alpha_n+\frac{1}{n}$. Let $n$ be
so large that $\alpha_n\ge\alpha-\frac{1}{n}$, and let $m\ge
n$, we then obtain
\begin{align}
  E(Y_n-Y_m)^2
  &=2EY_n^2+2EY_m^2-E(Y_n+Y_m)^2\notag\\
  &=2EY_n^2+2EY_m^2-4E(\tfrac{1}{2}(Y_n+Y_m))^2\notag\\
  &\le2(\alpha_n+\tfrac{1}{n})+2(\alpha_m+\tfrac{1}{m})-4\alpha_n\notag\\
  &=2(\tfrac{1}{n}+\tfrac{1}{m})+2(\alpha_m-\alpha_n).
\end{align}
As $(\alpha_n)$ is convergent, it is Cauchy. Therefore, the above
shows that $(Y_n)$ is Cauchy in $\LLL^2$, therefore convergent,
proving the lemma.
\end{proof}

Lemma \ref{lemma:Mazur} may be seen as a combination of variants of
the following two classical results: Every bounded sequence in a reflexive Banach
space contains a weakly convergent subsequence (see Theorem 4.41-B of \cite{AET}), and every weakly convergent
sequence in a reflexive Banach space has a sequence of convex
combinations of its elements converging strongly to the weak limit (see
Theorem 3.13 of \cite{WR2}). In \cite{BSV}, an $\LLL^1$ version of
Lemma \ref{lemma:Mazur} is used to give a simple proof of the Doob-Meyer theorem,
building on the ideas of \cite{AJ} and \cite{KMR}.

The remainder of the article is organized as follows. In Section
\ref{section:Compensator}, we give our proof of the existence of the compensator,
and in Section \ref{section:QV}, we give our proof of the
existence of the quadratic variation. In Section
\ref{section:Discussion}, we discuss how these results may be used to
give a simplified account of the theory of stochastic integration with
respect to semimartingales. In particular, the account proposed
excludes the use of: the d\'{e}but theorem, the section theorems and the
Doob-Meyer theorem. Appendix \ref{section:Auxiliary} contains
auxiliary results which are needed in the main proofs.

\section{The existence of the compensator}

\label{section:Compensator}

In this section, we will show that for any c\`{a}dl\`{a}g adapted process $A$ with initial
value zero and paths of finite variation, locally integrable, there
exists a c\`{a}dl\`{a}g predictable process $\Pi^*_pA$ with initial value zero
and paths of finite variation, locally integrable, unique up to indistinguishability, such that $A-\Pi^*_pA$
is a local martingale. We refer to $\Pi^*_pA$ as the compensator of
$A$. The proofs will use some basic facts from the general theory of
processes, some properties of monotone convergence for c\`{a}dl\`{a}g
increasing mappings, and Lemma \ref{lemma:Mazur}. Essential for the
results are the results on the limes superior of discrete
approximations to the compensator, the proof of this is based on the
technique developed in \cite{AJ} and also applied in \cite{BSV}. Note
that as the existence of the compensator follows directly from the
Doob-Meyer theorem, see for example Section I.3b of \cite{JS}, the
interest of the proofs given in this section is that if we restrict
our attention to the compensator of a finite variation process instead
of a submartingale, the complicated uniform integrability arguments
applied in \cite{KMR} may be done away with, and furthermore we need
only an $\LLL^2$ subsequence principle and not an $\LLL^1$ subsequence
principle as in \cite{BSV}. We begin by recalling
some standard nomenclature and fixing our notation.

By $\aaf$, we denote the set of processes which are c\`{a}dl\`{a}g adapted and
increasing with initial value zero. For $A\in\aaf$, the limit
$A_\infty$ of $A_t$ for $t$ tending to infinity always exists in
$[0,\infty]$. We say that $A$ is integrable if $A_\infty$ is
integrable. The subset of integrable processes in $\aaf$ is denoted by
$\aai$. For $A\in\aaf$, we say that $A$ is locally integrable if there
exists a localising sequence $(T_n)$ such that $A^{T_n}\in\aai$. The
set of such processes is denoted by $\aal$. By $\vvf$, we denote the
set of processes which are c\`{a}dl\`{a}g adapted with initial value zero and
has paths of finite variation. For $A\in\vvf$, $V_A$ denotes the
process such that $(V_A)_t$ is the variation of $A$ over
$[0,t]$. $V_A$ is then an element of $\aaf$. For
$A\in\vvf$, we say that $A$ is integrable if $V_A$ is integrable, and we say that $A$ is locally integrable if $V_A$ is
locally integrable. The corresponding spaces of stochastic processes
are denoted by $\vvi$ and $\vvl$, respectively. By $\DD_+$, we
denote the set of nonnegative dyadic rationals,
$\DD_+=\{k2^{-n}|k\ge0,n\ge0\}$. The space of square-integrable
martingales with initial value zero is denoted by $\msq$. Also, we say that two processes $X$
and $Y$ are indistinguishable if their sample paths are almost surely
equal, and in this case, we say that $X$ is a modification of $Y$ and
vice versa. We say that a process $X$ is c\`{a}dl\`{a}g if it is
right-continuous with left limits, and we say that a process $X$ is
c\`{a}gl\`{a}d if it is left-continuous with right limits.

Our main goal in this section is to show that for any $A\in\vvl$, there is
a predictable element $\Pi^*_pA$ of $\vvl$, unique up to indistinguishability, such that $A-\Pi^*_pA$ is a
local martingale. To prove the result, we first establish the existence of the
compensator for some simple elements of $\vvl$, namely processes
of the type $\xi1_{\ldb T,\infty\ldb}$, where $T$ is a stopping time
with $T>0$, $\xi$ is bounded, nonnegative and $\FFF_T$ measurable and
$\ldb T,\infty\ldb=\{(t,\omega)\in\RR_+\times\Omega\mid T(\omega)\le
t\}$. After this, we apply monotone convergence arguments and localisation arguments to obtain the general
existence result.

\begin{lemma}
\label{lemma:SimpleCompensator}
Let $T$ be a stopping time with $T>0$ and let $\xi$ be nonnegative, bounded and $\FFF_T$
measurable. Define $A = \xi1_{\ldb T,\infty\ldb}$. $A$ is then an element of $\aai$, and there exists a
predictable process $\Pi^*_pA$ in $\aai$ such that
$A-\Pi^*_pA$ is a uniformly integrable martingale.
\end{lemma}

\begin{proof}
Let $t^n_k=k2^{-n}$ for $k,n\ge0$. We define
\begin{align}
  A^n_t &= A_{t^n_k} \textrm{ for }t^n_k\le t<t^n_{k+1}
\end{align}
and
\begin{align}
  B^n_t &= \sum_{i=1}^{k+1} E(A_{t^n_i}-A_{t^n_{i-1}}|\FFF_{t^n_{i-1}})
            \textrm{ for }t^n_k< t\le t^n_{k+1},
\end{align}
and $B^n_0=0$. Note that both $A^n$ and $B^n$ have initial value zero,
since $T>0$. Also note that $A^n$ is c\`{a}dl\`{a}g adapted
and $B^n$ is c\`{a}gl\`{a}d adapted. Put $M^n=A^n-B^n$. Note that
$M^n$ is adapted, but not necessarily c\`{a}dl\`{a}g or c\`{a}gl\`{a}d. Also note
that, with the convention that a sum over an empty index set is zero,
it holds that
\begin{align}
 A^n_{t^n_k}=A_{t^n_k} \quad\textrm{ and }\quad
 B^n_{t^n_k}=\sum_{i=1}^k E(A_{t^n_i}-A_{t^n_{i-1}}|\FFF_{t^n_{i-1}})
\end{align}
for $k\ge0$. Therefore, $(B_{t^n_k})_{k\ge0}$ is the
compensator of the discrete-time increasing process
$(A_{t^n_k})_{k\ge0}$, see Theorem II.54 of \cite{RW1}, so
$(M^n_{t^n_k})_{k\ge0}$ is a discrete-time martingale with initial
value zero. We next show
that each element in this sequence of discrete-time martingales is bounded in
$\LLL^2$, and the limit variables constitute a sequence bounded in
$\LLL^2$ as well, this will allow us to apply Lemma \ref{lemma:Mazur}. To this
end, note that since $B^n$ has initial value zero,
\begin{align}
  (B^n_{t^n_k})^2
  &=2(B^n_{t^n_k})^2-\sum_{i=0}^{k-1}(B^n_{t^n_{i+1}})^2-(B^n_{t^n_i})^2\notag\\
  &=\sum_{i=0}^{k-1}2B^n_{t^n_k}(B^n_{t^n_{i+1}}-B^n_{t^n_i})-(B^n_{t^n_{i+1}})^2+(B^n_{t^n_i})^2\notag\\
  &=\sum_{i=0}^{k-1}2(B^n_{t^n_k}-B^n_{t^n_i})(B^n_{t^n_{i+1}}-B^n_{t^n_i})-(B^n_{t^n_{i+1}}-B^n_{t^n_i})^2\notag\\
  &\le\sum_{i=0}^{k-1}2(B^n_{t^n_k}-B^n_{t^n_i})(B^n_{t^n_{i+1}}-B^n_{t^n_i}).
\end{align}
Now let $c$ be a bound for $\xi$. Applying that $B^n_{t^n_{i+1}}$ is $\FFF_{t^n_i}$ measurable, the
martingale property of $(M^n_{t^n_k})_{k\ge0}$ and the fact that $A$
and $B$ are increasing and $A$ is bounded by $c$, we find
\begin{align}
  E(B^n_{t^n_k}-B^n_{t^n_i})(B^n_{t^n_{i+1}}-B^n_{t^n_i})
  &=E(B^n_{t^n_{i+1}}-B^n_{t^n_i})E(B^n_{t^n_k}-B^n_{t^n_i}|\FFF_{t^n_i})\notag\\
  &=E(B^n_{t^n_{i+1}}-B^n_{t^n_i})E(A^n_{t^n_k}-A^n_{t^n_i}|\FFF_{t^n_i})\notag\\
  &\le cE(B^n_{t^n_{i+1}}-B^n_{t^n_i}).
\end{align}
All in all, we find $E(B^n_{t^n_k})^2\le
2c\sum_{i=0}^{k-1}E(B^n_{t^n_{i+1}}-B^n_{t^n_i})=2cEB^n_{t^n_k}=2cEA^n_{t^n_k}\le
2c^2$. Thus $E(M^n_{t^n_k})^2\le
4E(A^n_{t^n_k})^2+4E(B^n_{t^n_k})^2\le 12c^2$. We conclude that
$(M^n_{t^n_k})_{k\ge0}$ is bounded in $\LLL^2$, and so convergent
almost surely and in $\LLL^2$ to a limit $M^n_\infty$, and the
sequence $(M^n_\infty)_{n\ge0}$ is bounded in $\LLL^2$ as well.

By Lemma \ref{lemma:Mazur}, there exists a sequence of
naturals $(K_n)$ with $K_n\ge n$ and for each $n$ a finite sequence of reals
$\lambda^n_n,\ldots,\lambda^n_{K_n}$ in the unit interval summing to
one, such that $\sum_{i=n}^{K_n}\lambda^n_i M^i_\infty$ is convergent
in $\LLL^2$ to some variable $M_\infty$. By Theorem II.70.2 of \cite{RW2}, there is $M\in\msq$ such that $E\sup_{t\ge0}(M_t-\sum_{i=n}^{K_n}\lambda^n_iM^i_t)^2$ tends
to zero, $M$ is a c\`{a}dl\`{a}g version of the process $t\mapsto
E(M_\infty|\FFF_t)$. By picking a subsequence and relabeling, we may
further assume that $\sup_{t\ge0}(M_t-\sum_{i=n}^{K_n}\lambda^n_iM^i_t)^2$ also converges
almost surely to zero. Define $B=A-M$, we wish to argue that there is
a modification of $B$ satisfying the requirements of the lemma.

First put $C^n = \sum_{i=n}^{K_n}\lambda^n_iB^i$. Note that $C^n$ is
c\`{a}dl\`{a}g, adapted and increasing, and
\begin{align}
  \lim_{t\to\infty}C^n_t
  &=\lim_{m\to\infty}C^n_m
    =\lim_{m\to\infty}\sum_{i=n}^{K_n}\lambda^n_iB^i_m\notag\\
  &=\lim_{m\to\infty}A_m-\sum_{i=n}^{K_n}\lambda^n_iM^i_m
    =A_\infty -\sum_{i=n}^{K_n}\lambda^n_iM^i_\infty,
\end{align}
showing that $C^n\in\aai$ and that $(C^n_\infty)_{n\ge0}$ is bounded in
$\LLL^2$. Also note that for each $q\in\DD_+$, it holds that $A_q = \lim_{n\to\infty} A^n_q$ almost
surely. Therefore,
\begin{align}
  B_q = A_q-M_q
      = \lim_{n\to\infty}A^n_q-\sum_{i=n}^{K_n}\lambda^n_iM^i_q
      = \lim_{n\to\infty}\sum_{i=n}^{K_n}\lambda^n_iB^i_q
      = \lim_{n\to\infty}C^n_q,
\end{align}
almost surely. From this, we obtain that $B$ is almost surely increasing on
$\DD_+$. As $B$ is c\`{a}dl\`{a}g, this shows that $B$ is almost surely
increasing on all of $\RR_+$. Next, we show that $B_t =
\limsup_{n\to\infty}C^n_t$ almost surely, simultaneuously for all
$t\ge0$, this will allow us to show that $B$ has a predictable modification. To this end, note that for $t\ge0$ and $q\in\DD_+$ with
$q\ge t$, $\limsup_{n\to\infty} C^n_t \le \limsup_{n\to\infty} C^n_q=
B_q$. As $B$ is c\`{a}dl\`{a}g, this yields $\limsup_{n\to\infty}C^n_t\le
B_t$. This holds almost surely for all $t\in\RR_+$ simultaneously.
Similarly, $\liminf_{n\to\infty}C^n_t\ge
B_{t-}$ almost surely, simultaneously
for all $t\ge0$. All in all, we conclude that almost surely, $B_t =
\limsup_{n\to\infty}C^n_t$ for all continuity points $t$ of $B$, simultaneously
for all $t\ge0$. As the jumps of $B$ can be exhausted by a
countable sequence of stopping times, we find that in order to show
the desired result on the limes superior, it suffices to show for any
stopping time $S$ that $B_S=\limsup_{n\to\infty} C^n_S$.

Fixing a stopping time $S$, we first note that as $0\le C_S^n\le
C^n_\infty$, the sequence of variables $(C_S^n)_{n\ge0}$ is bounded in $\LLL^2$ and thus in particular
uniformly integrable. Therefore, by Lemma \ref{lemma:ReverseUIFatou},
$\limsup_{n\to\infty} EC^n_S\le E\limsup_{n\to\infty} C^n_S\le
EB_t$. As $\limsup_{n\to\infty} C^n_S\le
B_t$ almost surely, we find that to show $\limsup_{n\to\infty}
C^n_S=B_S$ almost surely, it suffices to show that
$EC^n_S$ converges to $EB_S$, and to this end, it suffices to show
that $EB^n_S$ converges to $EB_S$. Now define $S_n$ by putting $S_n=\infty$
whenever $S=\infty$ and $S_n=t^n_k$ whenever $t^n_{k-1}<S\le
t^n_k$. $(S_n)$ is then a sequence of stopping times taking values in
$\DD_+$ and infinity and converging downwards to $S$, and
\begin{align}
  B^n_S =\sum_{k=0}^\infty B^n_{t^n_{k+1}}1_{(t^n_k<S\le t^n_{k+1})}
  =\sum_{k=0}^\infty B^n_{t^n_{k+1}}1_{(S_n = t^n_{k+1})} =B^n_{S_n}.
\end{align}
As $A$ is c\`{a}dl\`{a}g and bounded and $A^n_{S_n}=A_{S_n}$, the dominated convergence theorem allows
us to obtain
\begin{align}
  \lim_{n\to\infty} EB^n_S
  =\lim_{n\to\infty} EB^n_{S_n}
  =\lim_{n\to\infty} EA_{S_n}-EM^n_{S_n}
  =\lim_{n\to\infty} EA_S-EM_S
  =EB_S.
\end{align}
Recalling our earlier observations, we may now conclude that $\limsup_{n\to\infty}C^n_t=B_t$
almost surely for all points of discontinuity of $B$, and so all in
all, the result holds almost surely for all $t\in\RR_+$
simultaneously.

We now apply this to show that $B$ has a predictable modification. Let
$F$ be the almost sure set where $B=\limsup_{n\to\infty}C^n$. Theorem 3.33 of
\cite{HWY} then shows that $1_FC^n$ is a predictable c\`{a}dl\`{a}g process,
and $B1_F=\limsup_{n\to\infty}C^n$. Therefore, $B1_F$ is a predictable
c\`{a}dl\`{a}g process, and $B$ is almost surely increasing as well. Now let $\Pi^*_pA$ be a
modification of $B$ such that $\Pi^*_pA$ is in $\aai$. Again using Theorem 3.33 of
\cite{HWY}, $\Pi^*_pA$ is predictable since $B$ is predictable, and as
$A-\Pi^*_pA$ is a modification of the uniformly integrable martingale $A-B$, we conclude that $\Pi^*_pA$
satisfies all the requirements to be the compensator of $A$.
\end{proof}

With Lemma \ref{lemma:SimpleCompensator} in hand, the remainder of
the proof for the existence of the compensator merely consists of
monotone convergence arguments.

\begin{lemma}
\label{lemma:CompensatorMonotone}
Let $A^n$ be a sequence of processes in $\aai$ such that
$\sum_{n=1}^\infty A^n$ converges pointwise to a process $A$. Assume
for each $n\ge1$ that $B^n$ is a predictable element of $\aai$
such that $A^n-B^n$ is a uniformly integrable martingale. $A$ is then in $\aai$, and
$\sum_{n=1}^\infty B^n$ almost surely converges pointwise to a
predictable process $\Pi^*_pA$ in $\aai$ such that $A-\Pi^*_pA$ is a
uniformly integrable martingale.
\end{lemma}

\begin{proof}
Clearly, $A$ is in $\aai$. With $B=\sum_{n=0}^\infty B^n$, $B$ is a well-defined process with
values in $[0,\infty]$, since each $B^n$ is nonnegative. We wish to argue
that there is a modification of $B$ which is the
compensator of $A$. First note that as each $B^n$ is increasing and nonnegative, so
is $B$. Also, as $A^n-B^n$ is a uniformly integrable martingale, the optional sampling
theorem and two applications of the monotone convergence theorem
yields for any bounded stopping time $T$ that
\begin{align}
  EB_T=\lim_{n\to\infty}\sum_{k=1}^nEB^k_T
       =\lim_{n\to\infty}\sum_{k=1}^n EA^k_T=EA_T,
\end{align}
which in particular shows that $B$ almost surely takes finite
values. Therefore, by Lemma \ref{lemma:IncreasingCadlagSumConv}, we
obtain that $B$ is almost surely nonnegative, c\`{a}dl\`{a}g and
increasing. Also, by another two applications of the monotone
convergence theorem, we obtain for any stopping time $T$ that
$EB_T=\lim_{t\to\infty}EB_{T\land t}=\lim_{t\to\infty}EA_{T\land
  t}=EA_T$. This holds in particular with $T=\infty$, and therefore,
the limit of $B_t$ as $t$ tends to infinity is almost surely
finite and is furthermore integrable. Lemma
\ref{lemma:IncreasingCadlagSumConv} then also shows that $\sum_{k=1}^n B^k$
converges almost surely uniformly to $B$ on $\RR_+$. 

We now let $\Pi^*_pA$ be a nonnegative c\`{a}dl\`{a}g increasing adapted
modification of $B$. Then $\Pi^*_pA$ is in $\aai$, and
$E(\Pi^*_pA)_T=EA_T$ for all stopping times $T$, so by Theorem 77.6 of
\cite{RW2}, $A-\Pi^*_pA$ is a uniformly integrable martingale. Also, $\sum_{k=1}^n B^k$ almost
surely converges uniformly to $\Pi^*_pA$ on $\RR_+$. In order to
complete the proof, it remains to show that $\Pi^*_pA$ is
predictable. To this end, note that by uniform convergence, Lemma
\ref{lemma:UniformCadlagConvergence} shows that for any stopping time
$T$, $\Delta (\Pi^*_pA)_T=\lim_n \sum_{k=1}^n \Delta B^k_T$. As $B^k$ is
predictable, we find by Theorem 3.33 of \cite{HWY} that if $T$ is
totally inacessible, $\Delta (\Pi^*_pA)_T$ is zero almost surely, and if $T$ is
predictable, $\Delta (\Pi^*_pA)_T$ is $\FFF_{T-}$ measurable. Therefore, Theorem
3.33 of \cite{HWY} shows that $\Pi^*_pA$ is predictable.
\end{proof}

\begin{theorem}
\label{theorem:CompensatorExistence}
Let $A\in\vvl$. There exists a predictable process $\Pi^*_pA$ in
$\vvl$, unique up to indistinguishability, such that $A-\Pi^*_pA$ is a local martingale.
\end{theorem}
\begin{proof}
We first consider uniqueness. If $A\in\vvl$ and $B$ and $C$ are two
predictable processes in $\vvl$ such that $A-B$ and $A-C$ both are
local martingales, we find that $B-C$ is a predictable local martingale with paths
of finite variation. By Theorem 6.3 of \cite{HWY}, uniqueness follows.

As for existence, Lemma \ref{lemma:SimpleCompensator} establishes
existence for the case where $A=\xi1_{\ldb T,\infty\ldb}$ where $\xi$
is nonnegative, bounded and $\FFF_T$ measurable. Using Lemma
\ref{lemma:CompensatorMonotone}, this extends to the case where
$\xi\in\LLL^1(\FFF_T)$. For general $A\in\aai$, there exists by
Theorem 3.32 of \cite{HWY} a sequence of
stopping times $(T_n)$ covering the jumps of $A$. Put
$A^d=\sum_{n=1}^\infty \Delta A_{T_n}1_{\ldb T_n,\infty\ldb}$. As
$A\in\aai$, $A^d$ is a well-defined element of $\aai$, and $A-A^d$ is a continuous element of
$\aai$. As we have existence for each $\Delta A_{T_n}1_{\ldb
  T_n,\infty\ldb}$, Lemma \ref{lemma:CompensatorMonotone} allows us to
obtain existence for $A$. Existence for $A\in\vvi$ is
then obtained by decomposing $A=A^+-A^-$, where
$A^+,A^-\in\aai$, and extends to $A\in\vvl$ by a localisation argument.
\end{proof}

From the characterisation in Theorem
\ref{theorem:CompensatorExistence}, the usual properties of the
compensator such as linearity, positivity, idempotency and commutation
with stopping, can then be shown.

\section{The existence of the quadratic variation}

\label{section:QV}

In this section, we will prove the existence of the quadratic
variation process for a local martingale by a reduction to the cases
of bounded martingales and martingales of integrable variation,
applying Lemma \ref{lemma:Mazur} to obtain existence for bounded
martingales. Apart from Lemma \ref{lemma:Mazur}, the proofs will also
use the fundamental theorem of local martingales as well as some
properties of martingales with finite variation. Our method of proof
is direct and is simpler than the methods employed in for example
\cite{OK} or \cite{JS}, where the quadratic covariation is defined
through the integration-by-parts formula and requires the construction
and properties of the stochastic integral.

\begin{lemma}
\label{lemma:BoundedQVExistence}
Let $M$ be a bounded martingale with initial value zero. There exists a process
$[M]$ in $\aai$, unique up to indistinguishability, such that
$M^2-[M]\in\msq$ and $\Delta [M]=(\Delta M)^2$. We call $[M]$ the quadratic variation process of $M$.
\end{lemma}
\begin{proof}
We first consider uniqueness. Assume that $A$ and $B$ are two
processes in $\aai$ such that $M^2-A$ and $M^2-B$ are in $\msq$ and
$\Delta A=\Delta B=(\Delta M)^2$. In particular, $A-B$ is a continuous
element of $\msq$ and has paths of finite variation, so Theorem 6.3 of
\cite{HWY} shows that $A-B$ is almost surely zero, such that
$A$ and $B$ are indistinguishable. This proves uniqueness. Next, we consider the existence of the process. Let $t^n_k=k2^{-n}$
for $n,k\ge0$, we then find
\begin{align}
  M^2_t &= \sum_{k=1}^\infty M_{t\land t^n_k}^2-M_{t\land t^n_{k-1}}^2\notag\\
        &= 2\sum_{k=1}^\infty M_{t\land t^n_{k-1}}(M_{t\land t^n_k}-M_{t\land t^n_{k-1}})+\sum_{k=1}^\infty (M_{t\land t^n_k}-M_{t\land t^n_{k-1}})^2,
\end{align}
where the terms in the sum are zero from a point onwards, namely for
such $k$ that $t^n_{k-1}\ge t$. Define $N^n_t = 2\sum_{k=1}^\infty M_{t\land t^n_{k-1}}(M_{t\land
  t^n_k}-M_{t\land t^n_{k-1}})$. Our plan for the proof is to show
that $(N^n)$ is a bounded sequence in $\msq$. This will allow us to apply Lemma
\ref{lemma:Mazur} in order to obtain some $N\in\msq$ which is the limit of appropriate convex combinations of
the $(N^n)$. We then show that by putting $[M]$ equal to a
modification of $M^2-N$, we obtain a process with the desired qualities.

We first show that $N^n$ is a martingale by applying Theorem 77.6 of \cite{RW2}. Clearly, $N^n$ is c\`{a}dl\`{a}g and adapted with
initial value zero, and so it suffices to prove that $N^n_T$ is integrable and that
$EN^n_T=0$ for all bounded stopping times $T$. To this end, note that as $M$ is bounded, there is
$c>0$ such that $|M_t|\le c$ for all $t\ge0$. Then $N^n_T$ is clearly integrable, as it is the sum of
finitely many terms each bounded by $4c^2$, and we have
\begin{align}
  EN^n_T&=E\sum_{k=1}^\infty M_{T\land t_{k-1}^n}(M_{T\land t^n_k}-M_{T\land t^n_{k-1}})\\
  &=\sum_{k=1}^\infty EM^T_{t_{k-1}^n}(M^T_{t^n_k}-M^T_{t^n_{k-1}})
  =\sum_{k=1}^\infty EM^T_{t_{k-1}^n}E(M^T_{t^n_k}-M^T_{t^n_{k-1}}|\FFF_{t_{k-1}^n}),\notag
\end{align}
where the interchange of summation and expectation is allowed, as the
only nonzero terms in the sum are for those $k$ such that
$t^n_{k-1}\le t$, and there are only finitely many such terms. As
$M^T$ is a martingale,
$E(M^T_{t^n_k}-M^T_{t^n_{k-1}}|\FFF_{t_{k-1}^n})=0$ by
optional sampling, so the above is
zero and $N^n$ is a martingale by Theorem 77.6 of \cite{RW2}. Next, we show that $N^n$ is bounded in
$\LLL^2$. Fix $k\ge1$, we first consider a bound for the second moment of
$N^n_{t^n_k}$. To obtain this, note that for $i<j$,
\begin{align}
  &E(M_{t^n_{i-1}}(M_{t^n_i}-M_{t^n_{i-1}}))(M_{t^n_{j-1}}(M_{t^n_j}-M_{t^n_{j-1}}))\notag\\
  &=E(M_{t^n_{i-1}}(M_{t^n_i}-M_{t^n_{i-1}})E(M_{t^n_{j-1}}(M_{t^n_j}-M_{t^n_{j-1}})|\FFF_{t^n_i})\notag\\
  &=E(M_{t^n_{i-1}}(M_{t^n_i}-M_{t^n_{i-1}})M_{t^n_{j-1}}E(M_{t^n_j}-M_{t^n_{j-1}}|\FFF_{t^n_i}),
\end{align}
which is zero, as $E(M_{t^n_j}-M_{t^n_{j-1}}|\FFF_{t^n_i})=0$, and
by the same type of argument, we obtain
$E(M_{t^n_i}-M_{t^n_{i-1}})(M_{t^n_j}-M_{t^n_{j-1}})=0$. In other
words, the variables are pairwisely orthogonal, and so
\begin{align}
  E(N^n_{t^n_k})^2
  &= E\left(\sum_{i=1}^k M_{t^n_{i-1}}(M_{t^n_i}-M_{t^n_{i-1}})\right)^2
  =\sum_{i=1}^k E\left(M_{t^n_{i-1}}(M_{t^n_i}-M_{t^n_{i-1}})\right)^2\notag\\
  &\le c^2\sum_{i=1}^k E(M_{t^n_i}-M_{t^n_{i-1}})^2
  =c^2 E\left(\sum_{i=1}^k M_{t^n_i}-M_{t^n_{i-1}}\right)^2
  = c^2EM_{t^n_k}^2,
\end{align}
which yields $\sup_{t\ge0}E(N^n_t)^2= \sup_{k\ge1} E(N^n_{t^n_k})^2\le
\sup_{k\ge1}c^2EM_{t^n_k}^2\le 4c^2EM_\infty^2$, and this is finite. Thus,
$N^n\in\msq$, and $E(N^n_\infty)^2=\lim_t E(N^n_t)^2\le
4c^2 EM_\infty^2$, so $(N^n_\infty)_{n\ge1}$ is bounded in $\LLL^2$.

Now, by Lemma \ref{lemma:Mazur}, there exists a sequence of
naturals $(K_n)$ with $K_n\ge n$ and for each $n$ a finite sequence of reals
$\lambda^n_n,\ldots,\lambda^n_{K_n}$ in the unit interval summing to
one, such that $\sum_{i=n}^{K_n}\lambda^n_iN^i_\infty$ is convergent
in $\LLL^2$ to some variable $N_\infty$. It then holds
that there is $N\in\msq$ such that $E\sup_{t\ge0}(N_t-\sum_{i=n}^{K_n}\lambda^n_iN^i_t)^2$ tends
to zero. By picking a subsequence and relabeling, we may assume
without loss of generality that we also have almost sure
convergence. Define $A=M^2-N$, we claim that there is a modification
of $A$ satisfying the criteria of the theorem.

To prove this, first note that as $M^2$ and $N$ are c\`{a}dl\`{a}g and
adapted, so is $A$. We want to show that $A$ is almost surely
increasing and that $\Delta A=(\Delta M)^2$ almost
surely. We first consider the
jumps of $A$. To prove that $\Delta A=(\Delta M)^2$ almost surely, it
suffices to show that $\Delta A=(\Delta M_T)^2$ almost surely for any bounded stopping time
$T$. Let $T$ be any bounded stopping time. Since $\sup_{t\ge0}(N_t-\sum_{i=n}^{K_n}\lambda^n_iN^i_t)^2$ converges to
zero almost surely, we find
\begin{align}
A_T &= M_T^2 - N_T
      = \lim_{n\to\infty} \sum_{i=n}^{K_n}\lambda^n_i (M_T^2 - N^i_T)\notag\\
      &= \lim_{n\to\infty} \sum_{i=n}^{K_n}\lambda^n_i
      \sum_{k=1}^\infty (M_{T\land t^i_k}-M_{T\land t^i_{k-1}})^2,
\end{align}
almost surely. Similarly,
\begin{align}
  \Delta A_T
  =\lim_{n\to\infty} \sum_{i=n}^{K_n}\lambda^n_i  \sum_{k=1}^\infty
  (M_{t\land t^i_k}-M_{t\land t^i_{k-1}})^2-(M_{(t\land
    t^i_k)-}-M_{(t\land t^i_{k-1})-})^2,
\end{align}
understanding that $M_{(t\land t^i_k)-}$ is the limit of $M_{s\land
  t^i_k}$ with $s$ tending to $t$ strictly from below, and similarly
for $M_{(t\land t^i_{k-1})-}$. Fix $i,k\ge0$. By inspection, if $t\le
t^i_{k-1}$ or $t>t^i_k$, it holds that $(M_{t\land t^i_k}-M_{t\land   t^i_{k-1}})^2-(M_{(t\land t^i_k)-}-M_{(t\land t^i_{k-1})-})^2$ is
zero. In the case where $t$ is such that $t^i_{k-1}<t\le t^i_k$, we
instead obtain
\begin{align}
  (M_{t\land t^i_k}-M_{t\land t^i_{k-1}})^2&=(M_t-M_{t^i_{k-1}})^2\\
  (M_{(t\land t^i_k)-}-M_{(t\land t^i_{k-1})-})^2&=(M_{t-}-M_{t^i_{k-1}})^2,
\end{align}
so that with $s(t,i)$ denoting the unique
$t^i_{k-1}$ such that $t^i_{k-1}<t\le t^i_k$, we have
\begin{align}
  \Delta A_T
  &=\lim_{n\to\infty} \sum_{i=n}^{K_n}\lambda^n_i  (M_T-M_{s(T,i)})^2-(M_{T-}-M_{s(T,i)})^2\notag\\
  &=\lim_{n\to\infty} \sum_{i=n}^{K_n}\lambda^n_i  (M_T^2-2M_TM_{s(T,i)}-M_{T-}^2+2M_{T-}M_{s(T,i)})\notag\\
  &=(\Delta M_T)^2+2\lim_{n\to\infty} \sum_{i=n}^{K_n}\lambda^n_i  (M_{T-}-M_{s(T,i)}).
\end{align}
Now, we always have $|s(T,i)-T|\le 2^{-i}$ and $s(T,i)<T$. Therefore, given $\eps>0$,
there is $n\ge1$ such that for all $i\ge n$,
$|M_{T-}-M_{s(T,i)}|\le\eps$. As the $(\lambda^n_i)_{n\le i\le K_n}$
are convex weights, we obtain for $n$ this large that
$|\sum_{i=n}^{K_n}\lambda^n_i  (M_{T-}-M_{s(T,i)})|\le \eps$. This
allows us to conclude that $\sum_{i=n}^{K_n}\lambda^n_i
(M_{T-}-M_{s(T,i)})$ converges pointwise to zero, and so $\Delta A_T = (\Delta
M_T)^2$ almost surely. Since this holds for any arbitrary stopping
time, we now obtain $\Delta A=(\Delta M)^2$ up to indistinguishability.

Next, we show that $A$ is almost surely increasing. Let
$\DD_+=\{k2^{-n}|k\ge0,n\ge 1\}$, then $\DD_+$ is dense in
$\RR_+$. Let $p,q\in\DD_+$ with $p\le q$, we will show that $A_p\le A_q$ almost
surely. There exists $j\ge1$ and naturals $n_p\le n_q$ such that $p=n_p2^{-j}$
and $q=n_q2^{-j}$. We know that $A_p = \lim_{n\to\infty} \sum_{i=n}^{K_n}\lambda^n_i  \sum_{k=1}^\infty (M_{p\land
  t^i_k}-M_{p\land t^i_{k-1}})^2$,
and analogously for $A_q$. For $i\ge j$, $p\land t^i_k = n_p2^{-j}\land k2^{-i}=(n_p2^{i-j}\land k)2^{-i}$, and analogously for $q\land
t^i_k$. Therefore, we obtain that almost surely,
\begin{align}
  \lim_{n\to\infty} \sum_{i=n}^{K_n}\lambda^n_i  \sum_{k=1}^\infty (M_{p\land t^i_k}-M_{p\land t^i_{k-1}})^2
  &=   \lim_{n\to\infty} \sum_{i=n}^{K_n}\lambda^n_i  \sum_{k=1}^{n_p2^{i-j}} (M_{t^i_k}-M_{t^i_{k-1}})^2\notag\\
  &\le \lim_{n\to\infty} \sum_{i=n}^{K_n}\lambda^n_i  \sum_{k=1}^{n_q2^{i-j}} (M_{t^i_k}-M_{t^i_{k-1}})^2\notag\\
  &=\lim_{n\to\infty} \sum_{i=n}^{K_n}\lambda^m_i  \sum_{k=1}^\infty (M_{q\land t^i_k}-M_{q\land t^i_{k-1}})^2,
\end{align}
allowing us to make the same calculations in reverse and conclude $A_p\le A_q$ almost surely. As $\DD_+$ is countable, we conclude that
$A$ is inceasing on $\DD_+$ almost surely, and as $A$ is c\`{a}dl\`{a}g, we
conclude that $A$ is increasing almost surely. Furthermore, as we have that
$A_\infty=M^2_\infty-N_\infty$ and both $M^2_\infty$ and $N_\infty$
are integrable, we conclude that $A_\infty$ is integrable.

Finally, let $F$ be the null set where $A$ is not increasing. Put
$[M]=A1_{F^c}$. As we all null sets are in $\FFF_t$
for $t\ge0$, $[M]$ is adapted as $A$ is adapted. Furthermore, $[M]$ is
c\`{a}dl\`{a}g, increasing and $[M]_\infty$ exists and is
integrable. As $M^2 - [M] = N+A1_F$, where $A1_F$ is
almost surely zero and therefore in $\msq$, we now have constructed a process
$[M]$ which is in $\aai$ such that $M^2-[M]$ is in $\msq$ and
$\Delta[M]=(\Delta M)^2$ up to indistinguishability. This concludes
the proof.
\end{proof}

\begin{theorem}
Let $M$ be a local martingale with initial value zero. There exists
$[M]\in\aaf$ such that $M^2-[M]$ is a local martingale with initial
value zero and $\Delta [M]=(\Delta M)^2$.
\end{theorem}
\begin{proof}
We first consider the case where $M=M^b+M^i$, where $M^b$ and $M^i$
both are local martingales with initial value zero, $M^b$ is bounded
and $M^i$ is of integrable variation. In this case, $\sum_{0<s\le t}(\Delta M^i_t)^2$ is absolutely
convergent for any $t\ge0$, and we may therefore define a process
$A^i$ in $\aaf$ by putting $A^i_t=\sum_{0<s\le t}(\Delta M^i_t)^2$. As $M^b$ is bounded, $\sum_{0<s\le t}\Delta
M^b_t\Delta M^i_t$ is almost surely absolutely convergent as well, and so we
may define a process $A^x$ in $\vvf$ by putting $A^x_t = \sum_{0<s\le
  t}\Delta M^b_t\Delta M^i_t$. Finally, by Theorem
\ref{lemma:BoundedQVExistence}, there exists a process $[M^b]$ in
$\aai$ such that $(M^b)^2-[M^b]$ is in $\msq$ and $\Delta [M^b]=(\Delta M^b)^2$. We put $A_t =
[M^b]_t +2A^x+A^i$ and claim that there is a modification of $A$ satisfying the criteria in the
theorem.

To this end, first note that $A$ clearly is c\`{a}dl\`{a}g adapted of finite
variation, and for $0\le s\le t$, we have $[M^b]_t\ge
[M^b]_s+\sum_{s<u\le t}(\Delta M^b_u)^2$ almost surely, so that we obtain $A_t-A_s\ge\sum_{s<u\le t}(\Delta M^b_u+\Delta M^i_u)^2$ almost
surely, showing that $A$ is almost surely increasing. To show that
$M^2-A$ is a local martingale, note that
\begin{align}
  M^2-A =(M^b)^2-[M^b]+2(M^bM^i-A^x)+(M^i)^2-A^i.
\end{align}
Here, $(M^b)^2-[M^b]$ is in $\msq$ by Theorem
\ref{lemma:BoundedQVExistence}, in particular a local martingale. By the
integration-by-parts formula, we have $(M^i)^2_t-A^i_t=2\int_0^t
M^i_{s-}\df{M^i}_s$, where the integral is well-defined as $M_{s-}$ is
bounded on compacts. Using Theorem 6.5 of \cite{HWY}, the integral process $\int_0^t M^i_{s-}\df{M^i_s}$ is a local martingale, and so
$(M^i)^2-A^i$ is a local martingale. Therefore, in order to obtain that $M^2-A$
is a local martingale, it suffices to show that $M^bM^i-A^x$ is a
local martingale. By Theorem 5.32 of \cite{HWY}, $M^b_tM^i_t-\int_0^t
M^b_s\df{M^i}_s$ is a local martingale, so it suffices to show that $\int_0^t
M^b_s\df{M^i}_s-A^x_t$ is a local martingale. As $\Delta M^b$ is bounded, it is
integrable, and so we have
\begin{align}
 \int_0^t M^b_s\df{M^i}_s&=\int_0^t \Delta
M^b_s\df{M^i}_s+\int_0^t M^b_{s-}\df{M^i}_s=A^x_t+\int_0^t M^b_{s-}\df{M^i_s}.
\end{align}
As $\int_0^t M^b_{s-}\df{M^i_s}$ is a local martingale, again by Theorem 6.5 of \cite{HWY}, we finally conclude that
$M^bM^i-A^x$ is a local martingale. Thus, $M^2-A$ is a local martingale. This proves
existence in the case where $M=M^b+M^i$, where $M^b$ is bounded and
$M^i$ has integrable variation.

Finally, we consider the case of a general local martingale $M$ with
initial value zero. By Theorem III.29 of \cite{PP}, $M=M^b+M^i$, where $M^b$ is locally bounded and
$M^i$ has paths of finite variation. With $(T_n)$ a localising
sequence for both $M^b$ and $M^i$, our previous results then show the
existence of a process $A^n\in\aaf$ such that $(M^{T_n})^2-[M]$ is a
local martingale and $\Delta A^n = (\Delta M^{T_n})^2$. By uniqueness, we may define
$[M]$ by putting $[M]_t = A^n_t$ for $t\le T_n$. We then obtain
that $[M]\in\aaf$, $M^2-[M]$ is a local martingale and $\Delta[M]=(\Delta M)^2$, and the
proof is complete.
\end{proof}

\section{Discussion}

\label{section:Discussion}

The results given in Sections \ref{section:Compensator} and
\ref{section:QV} yield comparatively simple proofs of existence of the
compensator and the quadratic variation, two technical concepts
essential to martingale theory in general and stochastic calculus in
particular. We will now discuss how these proofs may be used to give a
simplified account of the development of the basic results of
stochastic integration theory. Specifically, the question we ask is
the following: How can one, starting from basic continuous-time martingale
theory, construct the stochastic integral of a locally bounded
predictable process with respect to a semimartingale, as simply as possible?

Since the publication of one of the first complete accounts of the
general theory of stochastic integration in \cite{DM}, several others
have followed, notably \cite{HWY}, \cite{RW2}, \cite{OK}, \cite{JS} and
\cite{PP}, each contributing with simplified and improved proofs. The accounts in \cite{HWY} and \cite{RW2} make use of the predictable
projection to prove the Doob-Meyer theorem, and to obtain the
uniqueness of this projection, they apply the difficult section
theorems. In \cite{OK} and \cite{PP}, this dependence is removed,
using the methods of, among others, \cite{KMR} and \cite{RFB}, respectively. In
general, however, the methods in \cite{OK} and \cite{PP} are not
entirely comparable, as \cite{OK} follows the traditional path of
starting with continuous-time martingale theory, developing some
general theory of processes, and finally constructing the stochastic
integral for semimartingales, while \cite{PP} begins by defining a
semimartingale as a ``good integrator'' in a suitable sense, and develops the theory from
there, in the end proving through the Bichteler-Dellacherie theorem
that the two methods are equivalent. The developement of the
stochastic integral we will suggest below follows in the tradition
also seen in \cite{OK}.

We suggest the following path to the construction of the stochastic integral:

\begin{enumerate}
\item Development of the predictable $\sigma$-algebra and predictable
  stopping times, in particular the equivalence between, in the
  notation of \cite{RW2}, being ``previsible'' and being ``announceable''.
\item Development of the main results on predictable processes, in
  particular the characterization of predictable c\`{a}dl\`{a}g processes as
  having jumps only at predictable times, and having the jump at a
  predictable time $T$ being measurable with respect to the
  $\sigma$-algebra $\FFF_{T-}$.
\item Proof of the existence of the compensator, leading to the
  fundamental theorem of local martingales, meaning the decomposition
  of any local martingale into a locally bounded and a locally
  integrable variation part. Development of the quadratic variation
  process using these results.
\item Construction of the stochastic integral using the fundamental
  theorem of local martingales and the quadratic variation process.
\end{enumerate}

The proofs given in Sections \ref{section:Compensator} and
\ref{section:QV} help make this comparatively short path possible. We
now comment on each of the points above, and afterwards compare the
path outlined with other accounts of the theory.

As regards point 1, the equivalence between a stopping time being previsible (having a predictable graph)
and being announceable (having an announcing sequence) is proved in
\cite{RW2} as part of the PFA theorem, which includes the introduction
of $\FFF_{T-}$. However, the equivalence between P (previsibility) and
A (accessibility) may be done without any reference to $\FFF_{T-}$, and this
makes for a pleasant separation of concerns.

The main result in point 2, the characterization of predictable c\`{a}dl\`{a}g
functions, can be found for example as Theorem 3.33 of \cite{HWY}. The
proof, however, implicitly uses the d\'{e}but theorem, which is almost as
difficult to obtain as the section theorems. However, the dependence
may be removed if only one can prove, without use of the d\'{e}but theorem, that the jumps of predictable
c\`{a}dl\`{a}g processes may be covered by a sequence of predictable times, and this
is in fact possible.

The existence of the compensator in point 3 may now be obtained as in
Section \ref{section:Compensator}, and the fundamental
theorem of local martingales may then be proven as in the proof of Theorem III.29
of \cite{PP}. After this, the existence of the quadratic variation may
be obtained as in Section \ref{section:QV}. Note that the traditional
method for obtaining the quadratic variation is either as the
remainder term in the integration-by-parts formula (as in \cite{JS}),
or through a localisation to $\msq$, applying the Doob-Meyer
theorem. Our method removes the need for the application of the
Doob-Meyer theorem.

Finally, in point 4, these results may be combined to obtain the
existence of the stochastic integral of a locally bounded predictable
process with respect to a semimartingale using the fundamental theorem
of local martingales and a modification of the methods given in
Chapter IX of \cite{HWY}.

As for comparisons of the approach outlined above with other approaches,
for example \cite{OK}, the main benefit of the above approach would be
that the development of the compensator is obtained in a very simple
manner, in particular not necessitating a decomposition into
predictable and totally inaccessible parts, and without any reference
to ``naturality''. Note, however, that the expulsion of ``naturality''
from the proof of the Doob-Meyer theorem in \cite{KMR} already was
obtained in \cite{AJ} and \cite{BSV}. In any case, focusing attention
on the compensator instead of a general supermartingale decomposition
simplifies matters considerably. Furthermore, developing the quadratic
variation directly using the fundamental theorem of local martingales allows
for a very direct construction of the stochastic integral, while the
method given in \cite{OK} first develops a preliminary integral for
local martingales which are locally in $\msq$.

\appendix

\section{Auxiliary results}

\label{section:Auxiliary}


\begin{lemma}
\label{lemma:ReverseUIFatou}
Let $(X_n)$ be a sequence of uniformly integrable variables. It then holds that
\begin{align}
  \limsup_{n\to\infty}EX_n\le E\limsup_{n\to\infty} X_n.
\end{align}
\end{lemma}
\begin{proof}
Since $(X_n)$ is uniformly integrable, it holds that
$\lim_{\lambda\to\infty}\sup_nEX_n1_{(X_n>\lambda)}$ is
zero. Let $\eps>0$ be given, we may then pick $\lambda$ so large that
$EX_n1_{(X_n>\lambda)}\le\eps$ for all $n$. Now, the sequence
$(\lambda-X_n1_{(X_n\le \lambda)})_{n\ge1}$ is nonnegative, and Fatou's lemma
therefore yields
\begin{align}
     \lambda-E\limsup_{n\to\infty}X_n1_{(X_n\le\lambda)}
  &=E\liminf_{n\to\infty}(\lambda-X_n1_{(X_n\le \lambda)})\notag\\
  &\le \liminf_{n\to\infty}E(\lambda-X_n1_{(X_n\le \lambda)})\notag\\
  &=\lambda-\limsup_{n\to\infty}EX_n1_{(X_n\le \lambda)}.
\end{align}
The terms involving the limes superior may be infinite and are
therefore a priori not amenable to arbitrary arithmetic manipulation. However,
by subtracting $\lambda$ and multiplying by minus one,
we yet find
\begin{align}
  \limsup_{n\to\infty}EX_n1_{(X_n\le \lambda)}\le
  E\limsup_{n\to\infty}X_n1_{(X_n\le\lambda)}.
\end{align}
As we have ensured that $EX_n1_{(X_n>\lambda)}\le\eps$ for all $n$, this yields
\begin{align}
       \limsup_{n\to\infty} EX_n
  \le \eps+E\limsup_{n\to\infty}X_n1_{(X_n\le\lambda)}
  \le \eps+E\limsup_{n\to\infty}X_n,
\end{align}
and as $\eps>0$ was arbitrary, the result follows.
\end{proof}

\begin{lemma}
\label{lemma:IncreasingCadlagSumConv}
Let $(f_n)$ be a sequence of nonnegative increasing c\`{a}dl\`{a}g mappings from
$\RR_+$ to $\RR$. Assume that $\sum_{n=1}^\infty f_n$ converges
pointwise to some mapping $f$ from $\RR_+\to\RR$. Then, the convergence is uniform on compacts, and $f$ is a nonnegative
increasing c\`{a}dl\`{a}g mapping. If $f(t)$ has a limit as $t$ tends
to infinity, the convergence is uniform on $\RR_+$.
\end{lemma}
\begin{proof}
Fix $t\ge0$. For $m\ge n$, we have
\begin{align}
  \sup_{0\le s\le t}\left|\sum_{k=1}^m f_k(s)-\sum_{k=1}^nf_k(s)\right|
  = \sup_{0\le s\le t} \sum_{k=n+1}^m f_k(s)
  = \sum_{k=n+1}^m f_k(t),
\end{align}
which tends to zero as $m$ and $n$ tend to infinity. Therefore,
$(\sum_{k=1}^n f_k)$ is uniformly Cauchy on $[0,t]$, and so has a c\`{a}dl\`{a}g
limit on $[0,t]$. As this limit must agree with the pointwise limit,
we conclude that $\sum_{k=1}^n f_k$ converges uniformly on compacts to
$f$, and therefore $f$ is nonnegative, increasing and c\`{a}dl\`{a}g.

It remains to consider the case where $f(t)$ has a limit $f(\infty)$
as $t$ tends to infinity. In this case, we find that $\lim_t
f_n(t)\le \lim_t f(t)=f(\infty)$, so $f_n(t)$ has a limit
$f_n(\infty)$ as $t$ tends to infinity as well. Fixing $n\ge1$, we have
\begin{align}
  \sum_{k=1}^n f_k(\infty)
  =\sum_{k=1}^n \lim_{t\to\infty}f_k(t)
  =\lim_{t\to\infty}\sum_{k=1}^n f_k(t)
  \le \lim_{t\to\infty} f(t)
  =f(\infty).
\end{align}
Therefore, $(f_k(\infty))$ is absolutely summable. As we have
\begin{align}
  \sup_{t\ge 0}\left|\sum_{k=1}^m f_k(t)-\sum_{k=1}^nf_k(t)\right|
  = \sup_{t\ge 0} \sum_{k=n+1}^m f_k(t)
  = \sum_{k=n+1}^m f_k(\infty),
\end{align}
we find that $(\sum_{k=1}^n f_k)$ is uniformly Cauchy on $\RR_+$, and
therefore uniformly convergent. As the limit must agree with the
pointwise limit, we conclude that $f_n$ converges uniformly to $f$ on
$\RR_+$. This concludes the proof.
\end{proof}

\begin{lemma}
\label{lemma:UniformCadlagConvergence}
Let $(f_n)$ be a sequence of bounded c\`{a}dl\`{a}g mappings from $\RR_+$ to
$\RR$. If $(f_n)$ is Cauchy in the uniform norm, there is a bounded
c\`{a}dl\`{a}g mapping $f$ from $\RR_+$ to $\RR$ such that
$\sup_{t\ge0}|f_n(t)-f(t)|$ tends to zero. In this case, it holds that
$\sup_{t\ge0}|f_n(t-)-f(t-)|$ and $\sup_{t\ge0}|\Delta f_n(t)-\Delta
f(t)|$ tends to zero as well.
\end{lemma}
\begin{proof}
Assume that $(f_n)$ is Cauchy in the uniform norm. This implies that
$(f_n(t))_{n\ge1}$ is Cauchy for any $t\ge0$, therefore
convergent. Let $f(t)$ be the limit. Now note that as $(f_n)$ is
Cauchy in the uniform norm, $(f_n)$ is bounded in the uniform norm,
and therefore $\sup_{t\ge0}|f(t)|\le
\sup_{n\ge1}\sup_{t\ge0}|f_n(t)|$, so $f$ is bounded as well. In order
to obtain uniform convergence, let $\eps>0$. Let $k$ be such that for
$m,n\ge k$, $\sup_{t\ge0}|f_n(t)-f_m(t)|\le\eps$. Fix $t\ge0$, we then
obtain for $n\ge k$ that
\begin{align}
  |f(t)-f_n(t)|=\lim_m|f_m(t)-f_n(t)|\le\eps.
\end{align}
Therefore, $\sup_{t\ge0}|f(t)-f_n(t)|\le\eps$, and so $f_n$ converges uniformly
to $f$.

We now show that $f$ is c\`{a}dl\`{a}g. Let $t\ge0$, we will show that $f$ is
right-continuous at $t$. Take $\eps>0$ and take $n$ so that 
$\sup_{t\ge0}|f(t)-f_n(t)|\le\eps$. Let $\delta>0$ be such that
$|f_n(t)-f_n(s)|\le\eps$ for
$s\in[t,t+\delta]$, then
\begin{align}
  |f(t)-f(s)|\le|f(t)-f_n(t)|+|f_n(t)-f_n(s)|+|f_n(s)-f_n(t)|\le3\eps
\end{align}
for such $s$. Therefore,
$f$ is right-continuous at $t$. Now let $t>0$, we claim that $f$ has a
left limit at $t$. First note that for $n$ and $m$ large enough, it
holds for any $t>0$ that
$|f_n(t-)-f_m(t-)|\le\sup_{t\ge0}|f_n(t)-f_m(t)|$. Therefore, the
sequence $(f_n(t-))_{n\ge1}$ is Cauchy, and so convergent to some
limit $\xi(t)$. Now let $\eps>0$ and take $n$ so that
$\sup_{t\ge0}|f(t)-f_n(t)|\le\eps$ and $|f_n(t-)-\xi(t)|\le\eps$. Let
$\delta>0$ be such that $t-\delta\ge0$ and such that whenever
$s\in[t-\delta,t)$, $|f_n(s)-f_n(t-)|\le\eps$. Then
\begin{align}
  |f(s)-\xi(t)|\le
  |f(s)-f_n(s)|+|f_n(s)-f_n(t-)|+|f_n(t-)-\xi(t)|\le3\eps
\end{align}
 for any such $s$. Therefore, $f$ has a left limit at $t$. This shows that $f$ is c\`{a}dl\`{a}g.

Finally, we have for any $t>0$ and any sequence $(s_n)$
converging strictly upwards to $t$ that
$|f(t-)-f_n(t-)|=\lim_m|f(s_m)-f_n(s_m)|\le
\sup_{t\ge0}|f(t)-f_n(t)|$, so we conclude that $\sup_{t\ge0}|f(t-)-f_n(t-)|$ converges
to zero as well. As a consequence, we also obtain that $\sup_{t\ge0}|\Delta f(t)-\Delta
f_n(t)|$ converges to zero.
\end{proof}

\end{document}